\documentclass[12pt]{amsart}
\usepackage[a4paper, hmargin=2.5cm, vmargin={3cm, 3cm}]{geometry}

\usepackage{graphicx}
\usepackage{amssymb, amsthm, amsmath}
\usepackage{xcolor}
\newtheorem{theorem}{Theorem}

\DeclareGraphicsExtensions{.pdf,.png,.jpg}
\usepackage{pict2e}

\newtheorem{lemma}[theorem]{Lemma}

\newtheorem{corollary}[theorem]{Corollary}

\newtheorem{prop}[theorem]{Proposition}

\newtheorem*{theorem*}{Theorem}{\bf}{\it}
\newtheorem*{proposition*}{Proposition}{\bf}{\it}
\newtheorem*{observation*}{Observation}{\bf}{\it}
\newtheorem*{lemma*}{Lemma}{\bf}{\it}

\theoremstyle{definition}

\theoremstyle{remark}
\newtheorem{remark}[theorem]{Remark}

\newcommand{\supp}{\text{supp}\;}

\renewcommand{\epsilon}{\varepsilon}
\renewcommand{\phi}{\varphi}

\newcommand{\Z}{\mathbb Z}
\newcommand{\R}{\mathbb R}

\renewcommand{\tilde}{\widetilde}

\definecolor{lev}{rgb}{0.773,0.294,0.549}

\def\XXint#1#2#3{{\setbox0=\hbox{$#1{#2#3}{\int}$ }
\vcenter{\hbox{$#2#3$ }}\kern-.6\wd0}}

\begin{document}
\title[]{ Eigenfunctions with infinitely many isolated critical points}

\author{Lev Buhovsky, Alexander Logunov, Mikhail Sodin}

\begin{abstract}
We construct a Riemannian metric on the $ 2 $-dimensional torus, such that for infinitely many eigenvalues of the Laplace-Beltrami operator, a corresponding eigenfunction has infinitely many isolated critical points.
A minor modification of our construction implies
that each of these eigenfunctions
has a level set with infinitely many connected components
(i.e., a linear combination of two eigenfunctions may have infinitely many nodal domains).
\end{abstract}
\maketitle

\section{Introduction and the main result}

Let $ (X, g) $ be a compact connected Riemannian manifold without boundary.
The Riemannian structure $ g $ on $ X $ defines the Laplace-Beltrami
operator $ \Delta_g $ on the space of smooth functions on $ X$. The operator $ \Delta_g $ admits a spectral decomposition via orthonormal basis of
$ L^2(X,\Omega_g) $ (where $ \Omega_g $ is the volume density on $ X $ induced by $ g $), consisting of smooth real-valued functions, and with corresponding real and non-negative eigenvalues: $ \Delta_g \phi_j + \lambda_j \phi_j = 0 $, where $ \lambda_0 = 0 < \lambda_1 \leqslant \lambda_2 \leqslant \ldots $, and $ \phi_0 \equiv 1 $.

The number of connected components of the set $ Crit(\phi_k) $ of the critical points of $ \phi_k $ is an important geometric characteristic
of the eigenfunction $\phi_k$. We denote it by $ N_{crit}(k) $. In \cite{Y}, Yau asked whether there exists a non-trivial asymptotic lower bound on $ N_{crit}(k) $. The following theorem proven by Jakobson and Nadirashvili in~\cite{J-N} shows the negative answer to this question:

\begin{theorem*}[Jakobson-Nadirashvili]
There exists a Riemannian metric on $ \mathbb{T}^2 $ and a sequence of eigenfunctions of the corresponding Laplace-Beltrami operator, such that the corresponding eigenvalues converge to infinity, but the number of
critical points of the eigenfunctions from the sequence remains bounded.
\end{theorem*}

Our result states that generally one cannot hope for an asymptotic upper bound on $ N_{crit}(k) $:

\begin{theorem} \label{t:main}
On $ \mathbb{T}^2 $ there exists a Riemannian metric and a sequence of eigenfunctions of the corresponding Laplace-Beltrami operator, with eigenvalues converging to infinity, such that each one of the eigenfunctions from the sequence has an infinite number of isolated critical points.
\end{theorem}

Similarly to the construction of Jakobson and Nadirashvili, we also will  construct a  Riemannian metric on $ \mathbb{T}^2 $
of the Liouville type. The ``punch-line argument'' in our proof of Theorem \ref{t:main} invokes the Brower's fixed point theorem.

Let us mention that Enciso and Peralta-Salas \cite{E-PS} constructed a smooth metric such that the first eigenfunction has arbitrarily
large number of isolated critical points. Results of opposite spirit were proven
by Polterovich and Sodin \cite{Po-S} and
Polterovich-Polterovich-Stojisavljevi\'c~\cite{PPS}.

It is worth to point out that, at present, we do not know what happens in the case when the Riemannian metric $g$ is real-analytic. In that case,
an eigenfunction must have a finite number of isolated critical points, however,
we do not know whether there exists an asymptotic
upper bound for the number of critical points in terms of the corresponding eigenvalue.

At last, we mention that minor modification of our construction implies
that each of the eigenfunctions in the statement of Theorem~\ref{t:main}
has a level set with infinitely many connected components. This feature is
interesting in view of the failure of the Courant nodal domains theorem for linear combinations of eigenfunctions, see 
Gladwell-Zhou~\cite{Gladwell-Zhou}, Arnold~\cite{Arn}, B\'erard-Helffer~\cite{Berard-Helffer},  B\'erard-Charron-Helffer~\cite{Berard-Charron-Helffer} and references therein.
We will outline the needed modification at the end of this note.

\section{Proof of Theorem \ref{t:main}}

\subsection{The metric and $ 1 $-dim reduction}

Consider the coordinates $ (x,y) \in \mathbb{T}^2 = (\R / 2\pi \Z)^2 $ on the torus. Our Riemannian metric $ g $ on $ \mathbb{T}^2 $ will be of the form $ ds^2 = Q(x) (dx^2 + dy^2) $, where $Q$ is a $C^\infty$ smooth, positive, $2\pi$ periodic function.
Then the Laplace-Beltrami operator is
$ \Delta_g = \frac{1}{Q(x)} \left( \frac{\partial^2}{\partial x^2} + \frac{\partial^2}{\partial y^2} \right) $.
Searching eigenfunctions $ \phi $ of the form $ \phi(x,y) = F(x)G(y) $ with $2\pi$-periodic functions $F$ and $G$,
the equation $ \Delta_g \phi + \lambda \phi = 0 $ becomes equivalent to the system
\begin{equation} \label{eq:split}
\begin{cases}
   F''(x) + (\lambda Q(x) - \mu) F(x) = 0 \\
   G''(y) + \mu G(y) = 0 \\
 \end{cases}
\end{equation}

\noindent for some $ \mu \in \R $. Thus we get a sequence of solutions $ G_m(y) = e^{imy} $ with $ \mu_m = m^2 $, where $ m \in \Z $, and for each such $ m $, the first equation in the system $(\ref{eq:split})$ becomes
\begin{equation} \label{eq:maineq}
   F''(x) + (\lambda Q(x) - m^2) F(x) = 0.
\end{equation}

For each given $ m $, the eigenvalue problem $(\ref{eq:maineq})$ admits
a spectral decomposition,
where we have a sequence of smooth solutions $ F = F_{m,1}, F_{m,2}, \ldots $
with corresponding eigenvalues $ \lambda_{m,1} \leqslant \lambda_{m,2} \leqslant \ldots \, $.
Going back to our original eigenvalue problem on the torus,
$ \Delta_g \phi + \lambda \phi = 0 $,
we have found a sequence of solutions $ F_{m,k}(x) \cos my $, $ F_{m,k}(x) \sin my $. It is easy to see that this sequence of solutions is complete in $ L^2(\mathbb T^2, Q{\rm d}x{\rm d}y) $, and hence gives a complete spectral decomposition of our original eigenvalue
problem on the torus we will not use this fact).

Our goal now is to find a smooth positive function $ Q \in C^\infty (\R / 2\pi \Z) $,
such that {\em  for infinitely many integers $ m $, there will exist a solution $ F_m$ of the equation $(\ref{eq:maineq})$
having infinitely many isolated critical points in $ \R / 2\pi \Z $}. For such $ Q $ and $ F_m $, the eigenfunction $ F_m(x) \cos my $ (as well as the
eigenfunction $ F_m(x) \sin my $) will have infinitely many isolated critical points.

We have reduced $2$-dimensional eigenvalue problem to a $ 1 $-dimensional problem, and for convenience, we will consider periodic functions on $ \R $, instead of functions on $ \R / 2\pi \Z $. Moreover, by rescaling, we may assume that the functions are $ 8 $-periodic rather than $ 2\pi $-periodic.

\subsection{The main observation} \label{subs:main-observation}

Assume that $ u \colon \R \rightarrow \R $ is a non-trivial solution of the differential equation $ u''(x) + K(x) u(x) = 0 $, where $ K\colon \R \rightarrow \R $ is some smooth function. Also assume that for some point $ x_* \in \R $, we have $ K(x_*) = 0 $, and moreover, $ K $ makes a large number of oscillations near $ x_* $, such that these oscillations decay very fast when we approach $ x_* $. To be more precise, we suppose that for some $$ x_* = x_N < x_{N-1} < \ldots < x_1 < x_* + \epsilon, $$ we have $ (-1)^i K|_{(x_{i+1},x_i)} < 0 $, and moreover, for every $ 1 \leqslant i \leqslant N-2 $,
$$
\left| \int_{x_{i+1}}^{x_i} K(x) \, dx \right|  \ {\rm is\ much\ larger\ than\ } \
\sum_{j=i+1}^{N-1} \left| \int_{x_{j+1}}^{x_j} K(x) \, dx \right|.
$$
Furthermore, assume that $ u'(x_*) = 0 $, that $ u $ is positive on the interval $ [x_*,x_*+\epsilon] $, and
the values of $ u $ on that interval are of the same order of magnitude.
We claim that under these assumptions, $ u $ has at least $N-2$ critical points on $ (x_*,x_*+\epsilon) $.

Indeed, for any $ 1 \leqslant i < N $, we have $$ u'(x_i) - u'(x_{i+1}) = \int_{x_{i+1}}^{x_{i}} u''(x) \, dx = - \int_{x_{i+1}}^{x_i} K(x) u(x) \, dx ,$$ and moreover we have $ u'(x_N) = u'(x_*) = 0 $. Hence, by our assumptions, for each $ 1 \leqslant i < N $, the sign of $ u'(x_i) = (u'(x_i) - u'(x_{i+1})) + \ldots + (u'(x_{N-1}) - u'(x_N)) $ is positive if $ i $ is  even, and is negative when $ i $ is odd. Therefore for each $ 1 \leqslant i < N $, $ u $ has a critical point
$ \xi_i \in (x_{i+1},x_{i}) $, which is moreover isolated since
$ u''(\xi_i) = - K(\xi_i)u(\xi_i) \neq 0 $. We conclude that $ u $ has at least $ N-2 $ isolated critical points on the interval $ (x_*,x_* + \epsilon) $.

We can also consider a more general setting where the function $ K $ has infinitely many rapidly decaying oscillations near some point
$ x_* $, and in this case we may conclude that $ u $ has infinitely many isolated critical points. More precisely, it is enough to assume the following:

\begin{enumerate}
 \item
 We have some $ x_* \in \R $ and a strictly decreasing sequence $ (x_i)_{i=1}^\infty $, such that $ \lim_{i \rightarrow \infty} x_i = x_* $.
 \item We have a smooth function $ K\colon \R \rightarrow \R $ such that
 $ (-1)^i K|_{(x_{i+1},x_{i})} < 0 $ for every $ i $, and
 \[
 \left| \int_{x_{i+1}}^{x_i} K(x) \, dx \right|  \ge C\,
 \sum_{j=i+1}^{\infty} \left| \int_{x_{j+1}}^{x_j} K(x) \, dx \right|
 \]
  with some numerical constant $C>1$.
 \item $ u : \R \rightarrow \R $ is a solution of $ u''(x) + K(x) u(x) = 0 $ with $ u(x_*) > 0 $ and $ u'(x_*) = 0 $.
 \end{enumerate}
\noindent Then $ u $ has infinitely many isolated critical points on a small interval $ (x_*, x_* + \delta) $.

We will apply this idea to the sequence of periodic functions
$K_i(x)=\lambda_i Q(x) - m_i^2$ and the corresponding periodic solutions
$u_i=F_{m_i}$ of the ODE~\eqref{eq:maineq}.

\subsection{The construction} \label{subs:construction}

Here we will construct a family of smooth periodic functions $q_{S, \tau}(x)$
having infinitely many ``cascades'' each consisting of an infinite number of rapidly
decaying oscillations.

\medskip\noindent {\bf a)}
Let $ \psi : \mathbb{R} \rightarrow [0,1] $ be a smooth even function such that $ \supp(\psi)\subset (-1,-1/4) \cup (1/4,1) $
and such that $ \psi = 1 $ on $ [-3/4,-1/2] \cup [1/2,3/4] $. For every $ t \in [0,\infty) $,
denote $ \psi_t (x) = \psi(4^t x) $ and then define the $2$-periodic function $ \phi_t(x) = 1 - \sum_{n \in \mathbb{Z}} \psi_t(x+2n) $.

\begin{figure}[h!]
\centering
\includegraphics[scale=0.8]{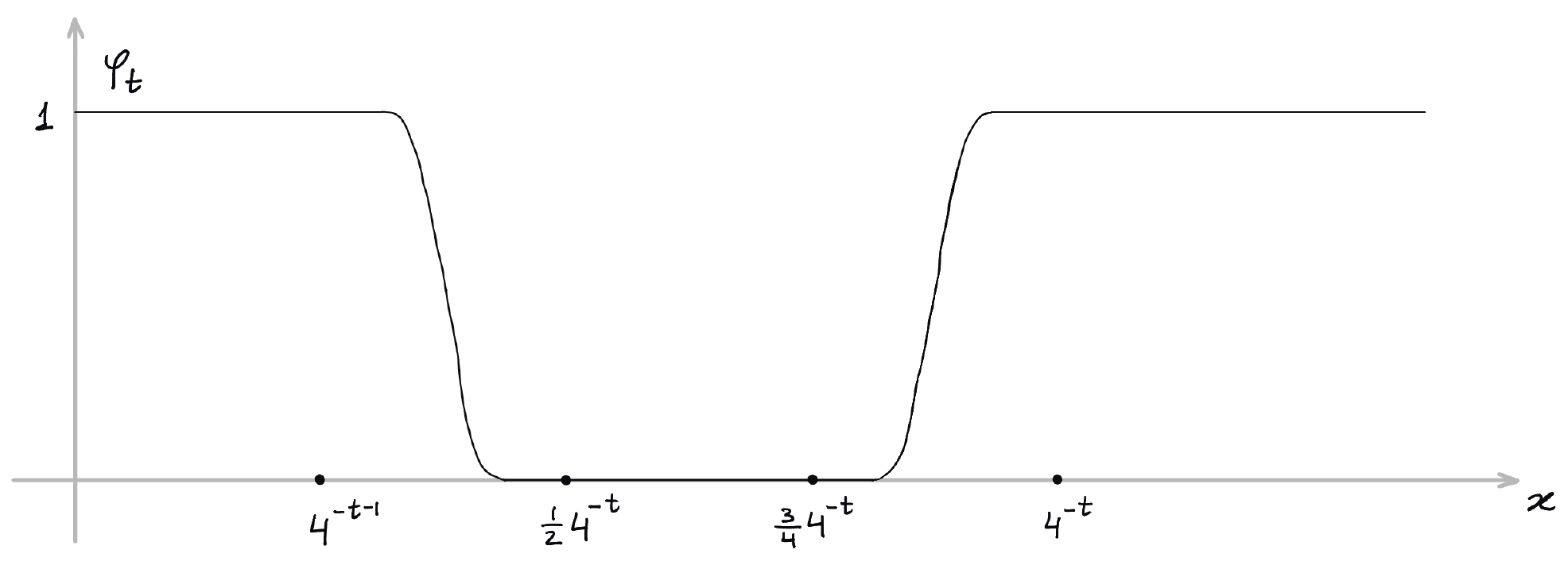}
\caption{\label{fig:graph-q-phi-t}}
\end{figure}

\medskip\noindent {\bf b)} Choose a smooth positive function $ q\colon \R \rightarrow (0, \infty) $ satisfying:

\begin{itemize}
\item $ q(x) = q(-x) $, $x\in\R$, \smallskip
\item $ q(x+2) = q(x)$,  $x \in \R $,\smallskip
\item $ q $ is increasing on $ [-1,0] $ and decreasing on $ [0,1] $.
\end{itemize}
For  $ t \in [0,\infty) $,  denote $q_t(x) = \phi_t(x) (q(x) - q(4^{-t})) + q(4^{-t}) $.
We may assume that the function $ q $ is ``flat''  enough at $ x=0 $ so that $ q_t \rightarrow q $ in the $ C^\infty $ topology, when $ t \rightarrow \infty $.
\begin{figure}[h!]
\centering
\includegraphics[scale=0.8]{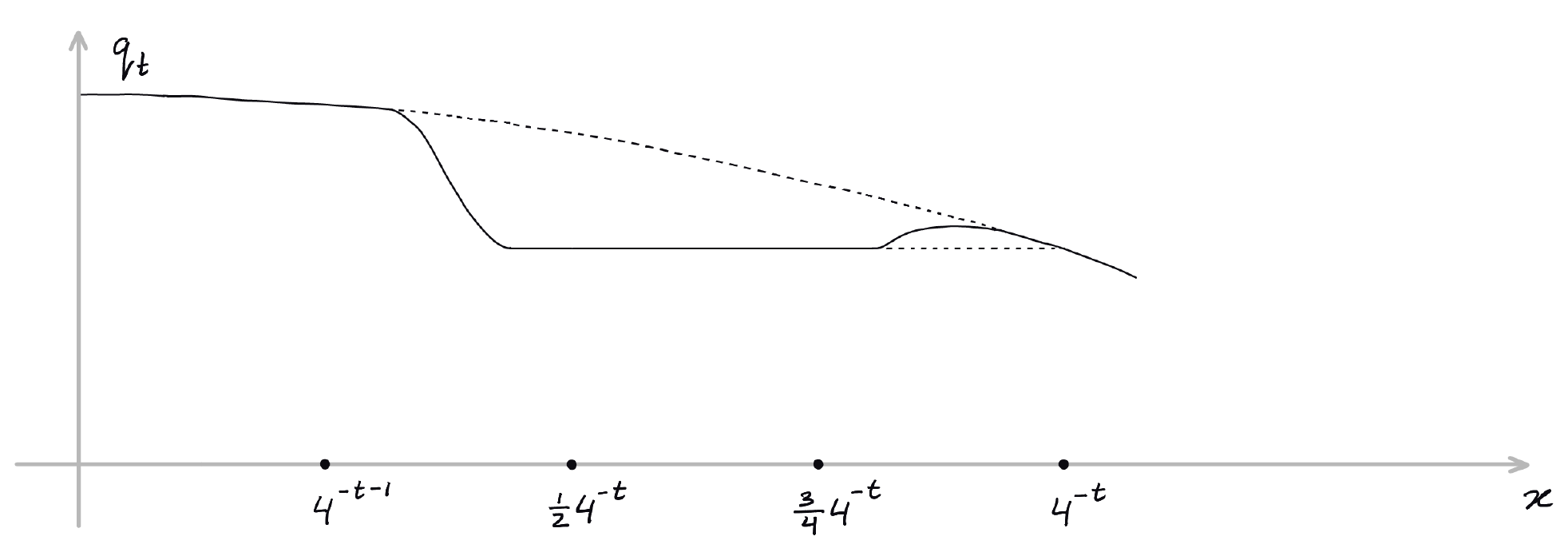}
\caption{\label{fig:graph-q-t}}
\end{figure}
This is possible since $q_t-q=(\phi_t-1)(q-q(4^{-t}))$. The first factor vanishes outside of
the interval $(4^{-t-1}, 4^{-t})$. Given a function $\phi_t$, we choose the function $q$ so flat at the origin that the smallness of $q-q(4^{-t})$ and its derivatives will
compensate the size of $\phi_t$ and its derivatives on the interval $[4^{-t-1}, 4^{-t}]$.

\medskip\noindent {\bf c)}
Fix a smooth even function $ h\colon \R \rightarrow \R $ with $ \supp h \subset (-3/4,-1/2) \cup (1/2,3/4) $, such that for some $ x_\infty \in (1/2,3/4)$ and for a strictly decreasing sequence $ (x_i)_{i=1}^\infty $ , $ x_i \in (x_\infty,3/4) $,
with $ \lim_{i \rightarrow \infty} x_i = x_\infty $, we have $ (-1)^i h|_{(x_{i+1},x_{i})} < 0 $
for every $ i $, and that moreover, for every $ i $ we have
\[
 \left| \int_{x_{i+1}}^{x_i} h(x) \, dx \right|  \ge C\,
 \sum_{j=i+1}^{\infty} \left| \int_{x_{j+1}}^{x_j} h(x) \, dx \right|
 \]
 with some numerical constant $C>1$.
For each $ t \in [0,\infty) $, denote $ h_t(x) = h(4^tx) $, and then $ H_t(x) = \sum_{n \in \Z} h_t(x+2n) $.

\begin{figure}[h!]
\centering
\includegraphics[scale=0.8]{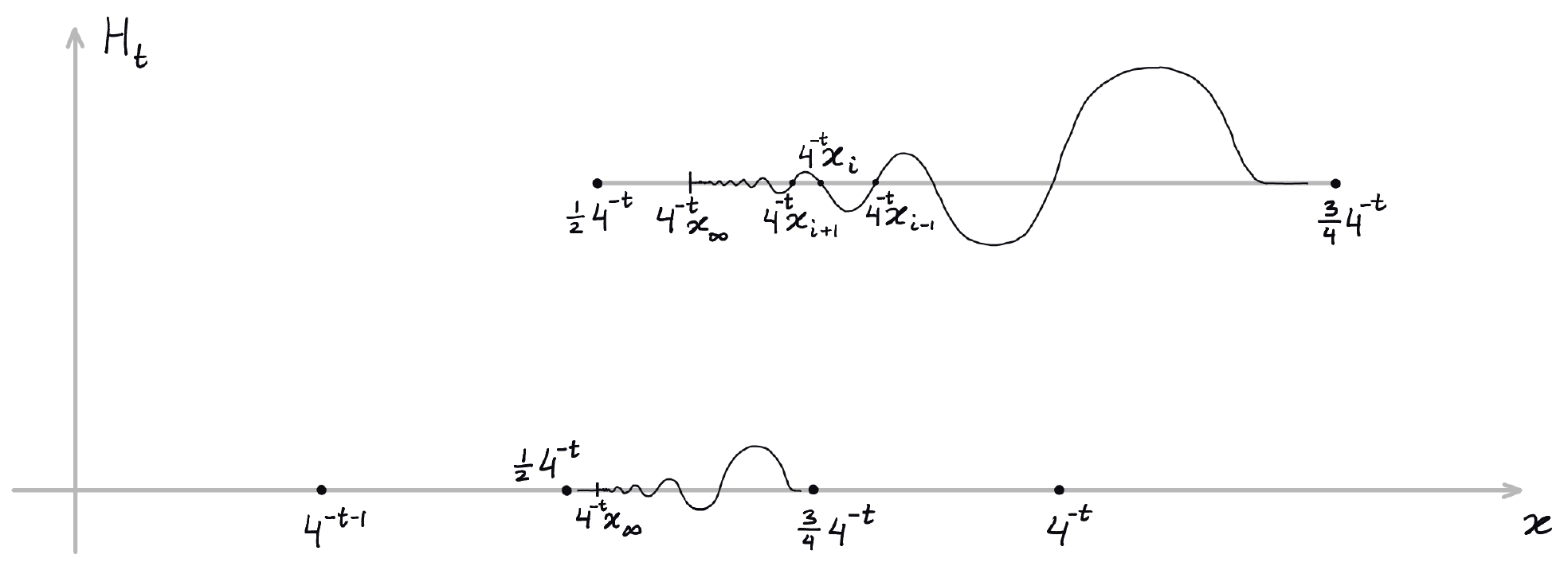}
\caption{\label{fig:graph-H-t}}
\end{figure}

\medskip\noindent {\bf d)}
For any finite $1$-separated set $ S \subset [0,\infty) $  ($1$-separation means that $ |s-t| \geqslant 1 $ for every $s\ne t \in S $),
and for any function $ \tau\colon S \rightarrow \R $, define the function
$ q_{S,\tau}\colon \R \rightarrow \R $ as follows. We set
$ q_{\varnothing, \tau} = q $, and if $ S \subset  [0,\infty) $ is a non-empty finite set, then we choose some $ t \in S $, denote $ S' = S \setminus \{ t \} $ and $ \tau' = \tau_{|S'} $, and define $ q_{S,\tau}(x) = \phi_t(x) (q_{S',\tau'}(x) - q_{S',\tau'}(4^{-t})) + q_{S',\tau'}(4^{-t}) + \tau(t) H_t(x)  $.
\begin{figure}[h!]
\centering
\includegraphics[scale=0.8]{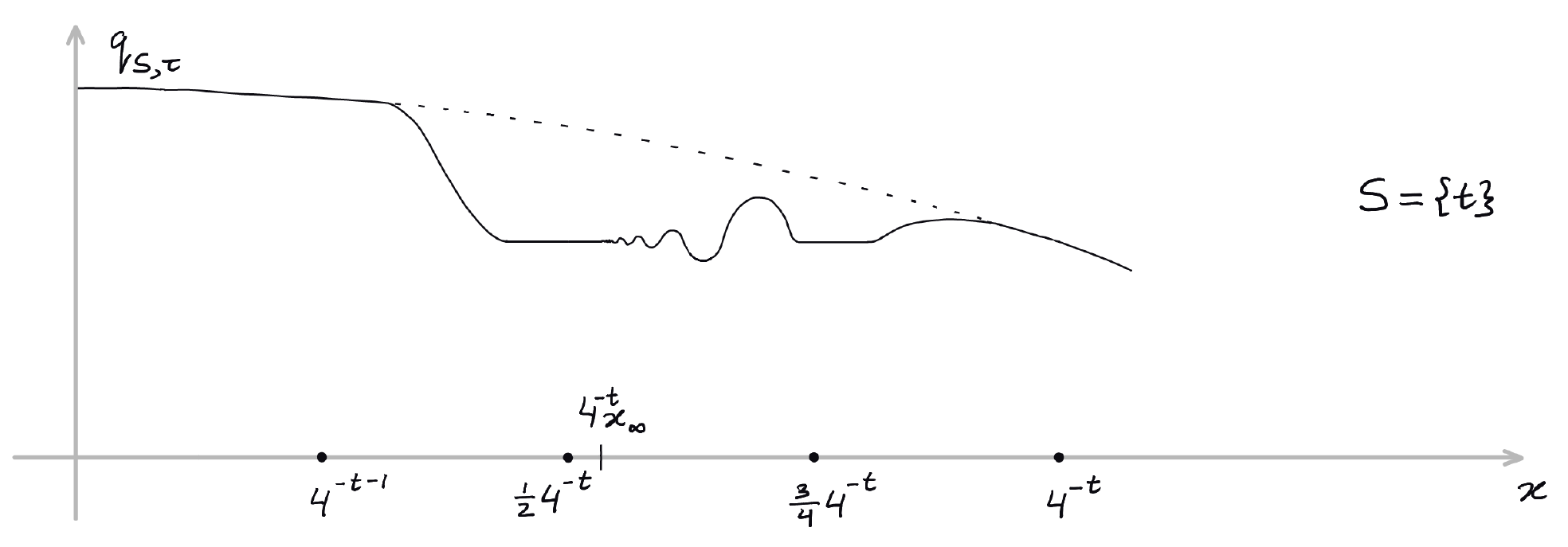}
\caption{\label{fig:graph-q-s-tau}}
\end{figure}
For each $t\in S$, this creates a cascade of oscillations around the point $4^{-t}x_\infty$ of intensity $\tau (t)$. Note that
$ \bigl\{x\colon q_{S, \tau}(x) \ne q(x) \bigr\}
\subset \bigcup_{t\in S} (4^{-t-1}, 4^{-t}) $ (the intervals on the RHS are disjoint),
and that, for $t\in S$, $q_{S, \tau}(4^{-t}x_\infty)=q(4^{-t})$.

\medskip\noindent {\bf e)}
Choose a smooth positive function $ \tilde q\colon  \R \rightarrow \R $ such that:

\begin{itemize}
\item $ \tilde q (x+2) = \tilde q (x)$, $  x \in \R $, \smallskip
\item $ \tilde q (-x) = \tilde q (x) $, $x \in \R $, \smallskip
\item $ \tilde q (x) = q (5x) $, $x \in [0,1/5] $, \smallskip

\item $ \tilde q (x) < q (4x) $, $x \in [1/5,2/9] $, \smallskip

\item $ \tilde q (x) < q (1) $, $x \in [2/9,1] $.
\end{itemize}

\begin{figure}[h!]
\centering
\includegraphics[scale=0.77]{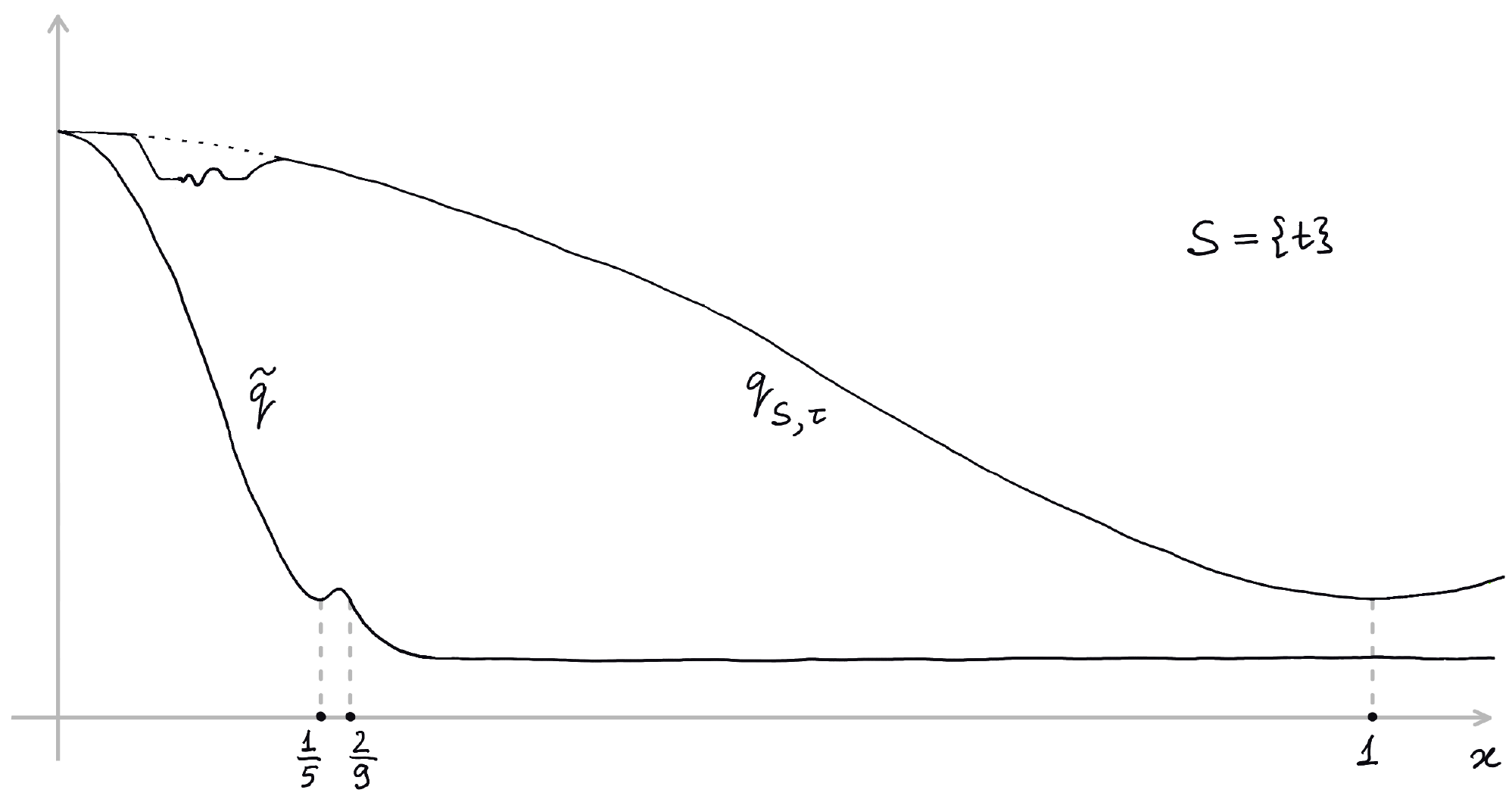}
\caption{\label{fig:graph-q-tilde}}
\end{figure}

We claim that {\em for any finite $1$-separated  set $ S \subset [1,\infty) $
and for any sufficiently small function $ \tau\colon S \rightarrow \R $,
we have $ \tilde q \leqslant q_{S,\tau} \leqslant q $}.
The upper bound is true for $ \tau $ small enough by the definition of the function $ q_{S,\tau} $.
To check the lower bound, it is enough to show that, for $ q_S := q_{S,0} $ (where
$ 0\colon S \rightarrow \R $ is the zero function),
we have $ q_S(x) > \tilde q (x) $ for any $ x \in (0,1] $.
Let us check this.

Note that, by the construction of $ q_S $, we have $ q_S(x) \geqslant q(4x) $ for $ x \in [0,1/4] $,
and $ q_S(x) \geqslant q(1) $ for $ x \in [1/4,1] $. Take any $ x \in (0,1] $. Then

\smallskip\noindent {\em \underline{Case 1:}}
$ \,\, x \in (0,1/5] $.
In that case we have $ q_S(x) - \tilde q(x) \geqslant q(4x) - \tilde q(x) = q(4x) - q(5x) > 0 $.

\smallskip\noindent {\em \underline{Case 2:}}
$ \,\, x \in [1/5,2/9] $. Then $ q_S(x) - \tilde q(x) \geqslant q(4x) - \tilde q(x) > 0 $.

\smallskip\noindent {\em \underline{Case 3:}}
$ \,\, x \in [2/9,1] $. Then $ q_S(x) - \tilde q(x) \geqslant q(1) - \tilde q(x) > 0 $.

\subsection{Two lemmas}

Here we will bring two lemmas on the solutions of the periodic ODE~\eqref{eq:maineq}.
The first lemma will guarantee that under a certain choice of $\lambda=\Lambda_{m, Q}$
the linear space of solutions to ODE~\eqref{eq:maineq} is two-dimensional. This will
ensure existence of a solution $F_m$ with vanishing derivative at a given point $x_*$,
which is needed for generation of oscillations of $F_m$ nearby $x_*$.

Let $ Q \colon \R \rightarrow (0, \infty) $
be a $ 2 $-periodic positive smooth even function. For every $ m \in \mathbb{N} $, consider the Dirichlet problem
\begin{equation} \label{eq:D-problem}
\begin{cases}
   u: [0,4] \rightarrow \R \\
   u''(x) + (\lambda Q(x) - m^2) u(x) = 0 \\
   u(0) = u(4) = 0 \\
  \end{cases}
\end{equation}

We denote by $ \Lambda_{Q,m} $ the first eigenvalue and by $ U_{Q,m} $ the first eigenfunction of this Dirichlet problem, extend $ U_{Q,m} $ to $ [0,8] $ by $ U_{Q,m}(x) = -U_{Q,m}(8-x) $ for $ x \in [4,8] $, and then extend $ U_{Q,m} $ to the whole $ \R $ by making it
$ 8 $-periodic. Then $ U_{Q,m}\colon \R \rightarrow \R $ is a smooth solution of
\begin{equation} \label{eq:N-problem}
\begin{cases}
   u\colon \R \rightarrow \R \text{ is 8-periodic} \\
   u''(x) + (\lambda Q(x) - m^2) u(x) = 0 \\
 \end{cases}
\end{equation}

\begin{lemma}\label{lemma:two-dim_space_of_solutions}
The space of solutions to $(\ref{eq:N-problem})$ with $ \lambda = \Lambda_{Q,m} $ is
two-dimensional.
\end{lemma}

\begin{proof}[Proof of Lemma~\ref{lemma:two-dim_space_of_solutions}]
Recall that $ Q $ is even and $2$-periodic. We have $ U_{Q,m}'(2) = 0 $,
since otherwise $ v(x) := U_{Q,m}(4-x) $ is another solution of $(\ref{eq:D-problem})$ with $ \lambda = \Lambda_{Q,m} $, which is linearly independent with $ U_{Q,m}|_{[0,4]} $, while the first eigenvalue of the Dirichlet problem $(\ref{eq:D-problem})$ must be simple. Therefore we get: $ U_{Q,m}(0) = 0 $, $ U_{Q,m}'(0) \neq 0 $, $ U_{Q,m}(2) \neq 0 $, $ U_{Q,m}'(2) = 0 $. Hence if we define $ V_{Q,m}= U_{Q,m}(x+2) $, then $ V_{Q,m}(x) $ is a solution of $(\ref{eq:N-problem})$ and is linearly independent with $ U_{Q,m} $. We conclude that the space of solutions of $(\ref{eq:N-problem})$ with $ \lambda = \Lambda_{Q,m} $ is
two-dimensional (obviously, the dimension cannot be greater than $ 2 $).
\end{proof}

In the second lemma we estimate the size of $\Lambda_{Q, m}$ chosen according to Lemma~\ref{lemma:two-dim_space_of_solutions}.

\begin{lemma} \label{l:varprin}
Let $ Q\colon \R \rightarrow (0, \infty) $ be a $ 2 $-periodic positive smooth function
with $ Q(0) > Q(x)$, $x \notin 2 \Z $. Then we have
\[ \Lambda_{Q,m} > \frac{m^2}{Q(0)} \]  for each $ m $,
and for every $ \epsilon > 0 $, we have
\[ \Lambda_{Q,m} < \frac{m^2}{Q(0) - \epsilon} \] when $ m $ is large enough.
\end{lemma}
\begin{proof}[Proof of Lemma \ref{l:varprin}]

The variational principle applied to the Dirichlet problem~\eqref{eq:D-problem}
(see, for instance,~\cite[Ch.~VI, \S1]{C-H}) says that
$$
\Lambda_{Q,m} = \min
\frac{\int_0^4 (u'(x))^2 + m^2(u(x))^2 \,
{\rm d}x}{\int_0^4 Q(x) (u(x))^2 \, {\rm d}x}\,,
$$
where the minimum is taken over all smooth functions $ u\colon [0,4] \rightarrow \R $ with $ u(0) = u(4) = 0 $. From here we clearly have
\[
\Lambda_{Q,m} > \frac{m^2}{\max Q} = \frac{m^2}{Q(0)}\,.
\]
Now, choose $ \delta > 0 $ small enough, and pick a smooth function
$ u\colon [0,4] \rightarrow \R $ with $ \supp(u) \subset (2-\delta,2+\delta) $
and $ u \not\equiv 0 $. Then for every $ m $ we have
$$
\Lambda_{Q,m} \leqslant
\frac{\int_0^4 (u'(x))^2 + m^2(u(x))^2 \, {\rm d }x}{\int_0^4 Q(x) (u(x))^2 \, {\rm d}x}
\leqslant
\frac{\int_0^4 (u'(x))^2 + m^2(u(x))^2 \, {\rm d}x}{(\min_{[2-\delta,2+\delta]} Q)\cdot
\int_0^4 (u(x))^2 \, {\rm d}x}
=
\frac{C+m^2}{\min_{[-\delta,\delta]} Q}\,,
$$
where
$$
C =  \frac{\int_0^4 (u'(x))^2 \, dx}{\int_0^4 (u(x))^2 \, dx} \,\, .
$$
From here we see that by choosing $ \delta > 0 $ so small that
$ \min_{[-\delta,\delta]} Q > Q(0) - \epsilon $,
we obtain
\[
\Lambda_{Q,m} < \frac{m^2}{Q(0) - \epsilon}
\]
for large enough $ m $.
\end{proof}

\subsection{The main argument}

We need to find an increasing sequence $(m_i)$ of positive
integers and a $1$-separated set $S=(t_i)$ so that for any fast decaying sequence
$\tau(t_i)$, letting $Q=q_{S, \tau}$, we get
$ \Lambda_{Q, m_i} Q(4^{-t_i}x_\infty) = m_i^2$, $i=1, 2, \, ...\,$,
where $\Lambda_{Q, m_i}$ is the first eigenvalue of the Dirichlet
problem $(\ref{eq:D-problem})$.
Note that $Q(4^{-t_i}x_\infty) = q_{S, \tau}(4^{-t_i}x_\infty) = q(4^{-t_i})$.

Using Lemma~\ref{l:varprin} and a continuous dependence of the first eigenvalue
$\Lambda_{Q, m}$ on the function $Q$, it is not difficult to find a
positive integer $m$ such that, for any
sufficiently small $\tau$, there exists $t$ (depending on $m$ and $\tau$ such that
$\Lambda_{q_{\{t\}, \tau}} q({4^{-t}}) = m^2$.
This will yield the existence of a metric on $\mathbb T^2$
with one Laplace eigenfunction with infinitely many critical points.
To construct a metric with a sequence of eigenfunctions having infinitely many critical
points we will use the following proposition.

\begin{prop} \label{p:main}
There exist sequences: $ m_1 < m_2 < \ldots $ of  positive integers,
$ \epsilon_1, \epsilon_2, \ldots  \in (0,\infty) $,
and $ M_0 = 0, M_1, M_2, \ldots \in [0,\infty) $, with $ M_{i+1} > M_i + 1 $
such that, for any $ k \in \mathbb{N} $ and any $ \tau_1,\tau_2, \ldots, \tau_k \in \R $
with $ |\tau_i| \leqslant \epsilon_i $, one can find $ t_1, \ldots, t_k $ with $ t_i \in [M_{i-1}+1,M_i] $, $i=1, 2, \ldots k$,  with the
following property:

For $ S = \{ t_1,t_2, \ldots, t_k \} $ and for the function
$ \tau\colon S \rightarrow \R $, $ \tau(t_i) = \tau_i $, $i=1,  2, \ldots , k$,
the function $ Q = q_{S,\tau} $ satisfies $ \Lambda_{Q,m_i} Q(4^{-t_i}x_\infty) - m_i^2 = 0 $ for $ i = 1, 2, \ldots, k $,
where $ \Lambda_{Q,m_i} $ is the first eigenvalue of the problem $(\ref{eq:D-problem})$ with $ m = m_i $.
\end{prop}

\begin{remark} \label{rem:x-infty}
As we have already mentioned,
by construction of the function $q_{S, \tau}$ in 2.3d, 
in the setting of the proposition we always have
$ Q(4^{-t_i}x_\infty) = q_{S, \tau}(4^{-t_i}x_\infty)= q(4^{-t_i}) $ for
$ i = 1, 2, \ldots, k $.
\end{remark}

\begin{proof}[Proof of Proposition \ref{p:main}]

First, we choose sequences $ (m_k) $, $ (\epsilon_k) $, $ (M_k) $ inductively as follows
(the values $t_1, \ldots , t_k$ will be chosen in the very end of the 
proof). 

\smallskip\noindent\underline{$k=1$}: By Lemma \ref{l:varprin}, there exists $ m_1 \in \mathbb{N} $ such that
\[
\frac{m_1^2}{q(0)} < \Lambda_{\tilde q,m_1} < \frac{m_1^2}{q(1/4)}\,.
\]
(Recall that $ \max \tilde q = \tilde q(0) = q(0) $). Then again by the lemma and by the monotonicity property for eigenvalues ($ q \geqslant \tilde q $), we have
\[
\frac{m_1^2}{q(0)} <
\Lambda_{q,m_1} \leqslant \Lambda_{\tilde q, m_1} < \frac{m_1^2}{q(1/4)}\,.
\]
 Let $ M_1 > 1 $ be such that
\[
 q(4^{-M_1}) = \frac{m_1^2}{\Lambda_{q,m_1}}\,.
\]
Now choose $ \epsilon_1 $ such that for any $ t_1 \in [1,M_1] $ and $ \tau_1 \in [-\epsilon_1,\epsilon_1] $ we have $ \tilde q \leqslant q_{S,\tau} \leqslant q $ on $ \mathbb{R} $, where $ S = \{ t_1 \} $, $ \tau : S \rightarrow \R $, $ \tau(t_1) = \tau_1 $.

\smallskip\noindent\underline{The inductive step}:
Assume that we have chosen
\[
m_1, \ldots, m_{k-1}, \quad \epsilon_1,\ldots, \epsilon_{k-1}, \quad M_1,\ldots,M_{k-1},
\]
and now choose $ m_k, \epsilon_k, M_k $. By Lemma \ref{l:varprin}, there exists $ m_k \in \mathbb{N} $ such that
\[
\frac{m_k^2}{q(0)} < \Lambda_{\tilde q,m_k} <
\frac{m_k^2}{q(4^{-1-M_{k-1}})}\,.
\]
Then, again by the lemma and by the monotonicity property for eigenvalues, we have
\[
\frac{m_k^2}{q(0)} < \Lambda_{q,m_k} \leqslant \Lambda_{\tilde q, m_k} < \frac{m_k^2}{q(4^{-1-M_{k-1}})}\,.
\]
 Now let $ M_k > M_{k-1} + 1 $ be such that
\[
q(4^{-M_k}) = \frac{m_k^2}{\Lambda_{q,m_k}}\,,
\]
and let $ \epsilon_k > 0 $ be such that for any $ t_k \in [M_{k-1} + 1, M_k] $ and any
$\tau_k \in [-\epsilon_k,\epsilon_k] $ we have $ \tilde q \leqslant q_{S,\tau} \leqslant q $ on $ \mathbb{R} $, where $ S = \{ t_k \} $ and $ \tau : S \rightarrow \R $, $ \tau(t_k) = \tau_k $.
This completes inductive construction of the sequences
$ (m_k) $, $ (\epsilon_k) $, $ (M_k) $. 

\medskip
Recall that $ M_0 = 0 $. Let $ k \in \mathbb{N} $, and let
\[
P := [M_0+1,M_1] \times [M_1+1,M_2] \times \dots \times [M_{k-1}+1,M_k]
\subset \mathbb{R}^k\,.
\]
Let $ \tau_1,\ldots ,\tau_k \in \mathbb{R} $ be such that $ |\tau_i| \leqslant \epsilon_i$,
$ i=1, 2, \ldots, k$.
Now define the map $ F : P \rightarrow \mathbb{R}^k $ as follows. For any $ (t_1,\ldots,t_k) \in P $, set $ S: = \{ t_1,\ldots,t_k \} $, let $ \tau\colon S \rightarrow \mathbb{R} $ be the function such that $ \tau(t_i) = \tau_i$, $i=1, 2, \ldots k$,
and denote $ Q = q_{S,\tau} $. We have $ \tilde q \leqslant Q \leqslant q $, hence for each $ 1 \leqslant i \leqslant k $,
$$ \frac{m_i^2}{q(4^{-M_i})} = \Lambda_{q,m_i} \leqslant \Lambda_{Q,m_i}  \leqslant \Lambda_{\tilde q,m_i} <  \frac{m_i^2}{q(4^{-1 - M_{i-1}})} \,,
$$
and therefore there exists a unique $ s_i \in (M_{i-1}+1,M_i] $ such that
\[
q(4^{-s_i}) = \frac{m_i^2}{\Lambda_{Q,m_i}}\,.
\]
Then we put $ F(t_1,\ldots,t_k) = (s_1,\ldots,s_k) $.

\medskip
Note first, that $ F $ is continuous (this follows from the continuous dependence of the
first Dirichlet eigenvalue $\Lambda_{Q, m}$ on the function $Q$). Secondly, for any $ (t_1,\ldots,t_k) \in P $ we have $ (s_1,\ldots,s_k) = F(t_1,\ldots,t_k) \in P $. That is, $ F(P) \subset P $. Therefore, by the Brower fixed point theorem, there exists a point $ (t_1,\ldots,t_k) \in P $ for which $ F(t_1,\ldots,t_k) = (t_1,\ldots,t_k) $, and hence by Remark \ref{rem:x-infty} we have
\[
Q(4^{-t_i}x_\infty) = q(4^{-t_i}) = \frac{m_i^2}{\Lambda_{Q,m_i}}
\]
for $ i = 1,2, \ldots, k$.
This finishes the proof of the proposition.
\end{proof}

\subsection{Completing the proof of Theorem~\ref{t:main}}

We fix the sequences $(m_i)$, $(M_i)$, and $(\epsilon_i)$ as in Proposition~\ref{p:main}.
For any $k\in\mathbb N$ and any $\tau_1, \ldots , \tau_k$ with $|\tau_i|\le \epsilon_i$,
Proposition~\ref{p:main} provides us with the values
$t_i(k)\in [M_{i-1}+1, M_i]$, $1\le i \le k$ such that, putting
$S(k)=\{t_1(k), \ldots , t_k(k)\}$, $Q_k=q_{S(k), \tau}$, we get
\[
\Lambda_{Q_k, m_i} Q_k(4^{-t_i}x_\infty) = m_i^2,
\quad i=1, \ldots , k.
\]
Now, we let $k\to\infty$. Using a diagonal argument, we assume that, for each
$i\in\mathbb N$, $t_i(k)\to t_i\in [M_{i-1}+1, M_i]$. Then we set $S=\{t_1, t_2, \ldots \}$.
Making the values $\tau_i$ sufficiently small, we note that
$ Q_k = q_{S(k), \tau} \to  q_{S, \tau} $ in the $C^\infty$-topology, and let
$Q=q_{S, \tau}$. At last, recalling that the first eigenvalue $\Lambda_{Q, m}$ of the Dirichlet problem~\eqref{eq:D-problem} continuously depends on $Q$,
we arrive at the following corollary:

\begin{corollary} \label{c:main}
There exists a sequence $ m_1,m_2, \ldots \in \mathbb{N} $, an infinite set
$ S = \{ t_1,t_2, \ldots \} \subset [1,\infty) $, with $ t_{i+1} \geqslant t_i + 1$, and a function $ \tau\colon S \rightarrow \R\setminus\{0\} $,
such that the function $ Q := q_{S,\tau} : \R \rightarrow \R $ is $ C^\infty $-smooth and
satisfies
\[
\Lambda_{Q,m_i} Q(4^{-t_i}x_\infty) - m_i^2 = 0, \quad i \in \mathbb{N}\,.
\]
\end{corollary}

Now we will readily finish off the proof of Theorem~\ref{t:main}.
Let $ Q $ be as in the corollary. Recall that the function $ h $ makes infinitely many rapidly decreasing oscillations near $ x_\infty $ (the function $ h $ and the point $ x_\infty $ were chosen in Section \ref{subs:construction}, c)). Hence for each $ i $, the function $ Q(x) $ makes infinitely many rapidly decreasing oscillations near $ 4^{-t_i} x_\infty $.

Fix $i$. 
Since the space of solutions of $(\ref{eq:N-problem})$ is $2$-dimensional
(by Lemma~\ref{lemma:two-dim_space_of_solutions}),
we can always find a solution of
\begin{equation} \label{eq:N-problem-subseq}
\begin{cases}
   u: \R \rightarrow \R \text{ is 8-periodic} \\
   u''(x) + (\Lambda_{Q,m_i} Q(x) - m_i^2) u(x) = 0, \\
 \end{cases}
\end{equation}
denoted by $ u_i(x) $, such that $ u_i( 4^{-t_i} x_\infty) > 0 $ and $ u_i'( 4^{-t_i} x_\infty) = 0 $. Applying our main observation (Section \ref{subs:main-observation}) with
$ K_{i}(x)
= \Lambda_{Q,m_i} Q(x) - m_i^2 $ and $ x_* = 4^{-t_i}x_\infty $, we conclude that $ u_i $ has infinitely many isolated critical points near $ 4^{-t_i} x_\infty $. This finishes the proof of Theorem \ref{t:main}. \hfill $\Box$

\section{Eigenfunctions with a level set having infinitely many connected components}

Here we outline changes in our construction needed to obtain a sequence of
eigenfunctions having a level set with infinitely many connected components.
For this, we replace condition (2) in Section~\ref{subs:main-observation} by a stronger
one:
\begin{itemize}
\item[($2'$)] $ K\colon \R \rightarrow \R $ is a smooth function such that
 $ (-1)^i K|_{(x_{i+1},x_{i})} < 0 $ for every $ i $, and
 \[
 \left| \int_{x_{i+1}}^{x_i} (x_i-x)K(x) \, dx \right|  \ge C\,
 \sum_{j=i+1}^{\infty} \left| \int_{x_{j+1}}^{x_j} K(x) \, dx \right|
 \]
  with some numerical constant $C>1$.
\end{itemize}
As above, we denote by $u\colon \mathbb R\to \mathbb R$ a solution to $u''+Ku=0$ with
$u(x_*)>0$, $u'(x_*)=0$. Choose $j_0$ so that $x_{j_0}-x_*<1$, and for every
$x, y\in [x_*, x_{j_0}]$, $1/C' \le u(x)/u(y) \le C'$ with $1<C'<C$.
Then it is not difficult to check that, for $j\ge j_0$, we have
\begin{itemize}
\item[(i)] ${\rm sign}( u'(x_j) )= (-1)^j$;
\item[(ii)] each interval $(x_{j+1}, x_j)$ contains exactly one critical point
$\xi_j$ of $u$; it is a local maximum of $u$ if $j$ is odd, and a local minimum
otherwise;
\item[(iii)] ${\rm sign}(u(x_j)-u(x_{j+1})) = (-1)^j$;
\item[(iv)] ${\rm sign}(u(\xi_j)-u(x_*)) = (-1)^{j-1}$.
\end{itemize}
We conclude that if $(\eta_j)_{j\ge j_0}$ are solutions to the equation
$u=u(x_*)$ on the interval $(x_*, x_{j_0}]$, then the sequences $(\eta_j)$ and $(\xi_j)$ interlace, i.e.,
$\eta_{j_0} > \xi_{j_0} > \eta_{j_0+1} > \xi_{j_0+1} > \eta_{j_0+2} > \ldots $.

Now, let the sequences $(m_i)$ and $(t_i)$, the function $Q$, and the values
$\Lambda_{Q, m_i}$ be the same as in Corollary~\ref{c:main}.
By $F_{m_i}$ we denote a solution of the $8$-periodic problem
$F'' + (\Lambda_{Q, m_i}Q-m_i^2) F = 0$ satisfying $F_{m_i}(4^{-t_i}x_\infty)>0$,
and $F_{m_i}'(4^{-t_i}x_\infty)=0$. Then the functions
$\phi_i(x, y) = F_{m_i}(x) \cos m_i y$ are Laplacian eigenfunctions on the torus
$\mathbb T^2$ equipped with the Riemannian metric
${\rm d}s^2 = Q(x)({\rm d}x^2 + {\rm d} y^2)$. Then, it is not difficult to see that,
for each of these eigenfunctions, the level sets $\phi_i = L_i$ with $L_i=F_{m_i}(4^{-t_i}x_\infty)$, have infinitely many connected components.

Indeed, fix $i$, take an odd $j\ge j_0(i)$, and consider the rectangle
$\mathcal R_j = \bigl\{\eta_{j+1}\le x \le \eta_j, |y|\le \pi/(2m_i) \bigr\}$.
Note that $\phi_i>L_i$ on the interval $\bigl\{ \eta_{j+1}<x<\eta_j, y=0 \bigr\}$,
and that $\phi_i=L_i$ at the end points of this interval. Furthermore, for each
$x\in [\eta_{j+1}, \eta_j]$, the function $y\mapsto \phi_i(x, y)$ decays to $0$
when $|y|$ increases from $0$ to $\pi/(2m_i)$. We conclude that, for each
odd $j\ge j_0(i)$, the set $\{\phi_i=L_i\} \cap \mathcal R_j$
is a topological circle. Clearly, these sets are disjoint for distinct odd
$j\ge j_0(i)$. We are done.

\subsection*{Acknowledgements}
The idea to construct the eigenfunction with many critical points explained in Section \ref{subs:main-observation} was suggested by Fedor Nazarov. We are grateful
to him for this suggestion and are sad that he refused to be a coauthor of this paper.
We thank the referee whose questions and comments helped us to improve the presentation.

L.B. was partially supported by ISF Grants 1380/13 and 2026/17, by the ERC Starting Grant 757585, and by Alon Fellowship. A.L. was partially supported
by ERC Advanced Grant 692616, ISF Grants 382/15 and 1380/13, during the time A.L. served as postdoctoral fellow at Tel Aviv University,
and part of this research was conducted during the time A.L. served as a Clay Research Fellow. M.S. was partially supported by ERC Advanced Grant 692616
and ISF Grant 382/15.

\vspace*{1cm}

{\noindent \bf Lev Buhovski,}$ $\\
School of Mathematical Sciences\\
Tel Aviv University \\
Ramat Aviv, Tel Aviv 69978\\
Israel\\
E-mail: levbuh@tauex.tau.ac.il \\

{\noindent \bf Alexander Logunov,}$ $\\
Department of Mathematics \\
Princeton University \\
Princeton, NJ, 08544 \\
USA \\
E-mail: log239@yandex.ru \\

{\noindent \bf Mikhail Sodin,}$ $\\
School of Mathematical Sciences\\
Tel Aviv University \\
Ramat Aviv, Tel Aviv 69978\\
Israel\\
E-mail: sodin@tauex.tau.ac.il


\begin{thebibliography}{ZZZZ}


\bibitem{Arn}
V. I. Arnold,
{\em Topological properties of eigenoscillations in mathematical physics},
Proc. Steklov Inst. Math. {\bf 273} (2011), 25--34.

\bibitem{Berard-Charron-Helffer}
P. B\'erard, P. Charron, and B. Helffer,
{\em Non-boundedness of the number of nodal domains of a sum of eigenfunctions},
{\tt arXiv:1906.03668}.

\bibitem{Berard-Helffer}
P. B\'erard, B. Helffer,
{\em On Courant's nodal domain property for linear combinations of eigenfunctions}.
Part~I {\tt arXiv:1705.03731}, Part~II {\tt arXiv:1803.00449}.

\bibitem{Gladwell-Zhou}
G. M. L. Gladwell, H. Zhu,
{\em The Courant-Herrmann conjecture},
ZAMM Z. Angew. Math. Mech. {\bf 83} (2003), 275--281.

\bibitem{C-H}
R. Courant and D. Hilbert, {\em Methods of Mathematical Physics}, Wiley, New York,
1989.

\bibitem{E-PS}
A. Enciso and D. Peralta-Salas, {\em Eigenfunctions with prescribed nodal sets}, J. Differential Geomemtry {\bf 101} (2015), 197--211.

\bibitem{J-N}
D. Jakobson, N. Nadirashvili, {\em Eigenfunctions with few critical points}, J. Differential Geometry {\bf 53} (1999), 177--182.

\bibitem{P}
A. Pleijel, {\em Remarks on Courant's nodal line theorem}, Comm. Pure Appl. Math. {\bf 9} (1956), 543--550.


\bibitem{PPS}
I. Polterovich, L. Polterovich, V. Stojisavljevi\'c,
{\em Persistence barcodes and Laplace eigenfunctions on surfaces},
{\tt arXiv:1711.07577}.

\bibitem{Po-S}
L. Polterovich and M. Sodin, {\em Nodal inequalities on surfaces}, Math. Proc. Cambridge
Philos. Soc. {\bf 143} (2007), 459--467.

\bibitem{Y}
S.-T. Yau, {\em Problem section, Seminar on Differential Geometry}, Annals of Mathematical Studies {\bf 102}, Princeton, 1982, 669--706.

\end{thebibliography}
\end{document}